\documentclass[10pt]{amsart}
\usepackage{amsmath,amsthm,amssymb,amsfonts, eucal, amscd, mathbbol, mathrsfs, mathabx}
\usepackage[shortlabels]{enumitem}
\usepackage{cite}
\usepackage{setspace}
\usepackage[all,cmtip]{xy}{
\usepackage{graphicx}
\usepackage{subfig}
\usepackage{fancyhdr}
\usepackage{latexsym}
\usepackage{fncylab}
\usepackage{epic}
\usepackage{ifthen}
\usepackage[bookmarks,bookmarksnumbered, plainpages=false, pdfpagelabels]{hyperref}
\usepackage{xcolor}
\hypersetup{
 colorlinks   = true, 
 urlcolor     = green, 
 linkcolor    = blue, 
 citecolor   = red 
}

\usepackage{rotating}
\usepackage{scalefnt}
\usepackage{enumitem}
\setlist{nolistsep}

\parskip = 0.2in
\parindent = 0.0in
\topmargin = 0.0in

\oddsidemargin = 0.0in
\evensidemargin = 0.0in
\textwidth = 6.0in

\newtheorem{thm}{Theorem}
\newtheorem*{thm*}{Main Theorem}
\newtheorem*{thm**}{Theorem}

\newtheorem{lem}[thm]{Lemma}

\theoremstyle{definition}

\theoremstyle{definition}

\theoremstyle{definition}

\theoremstyle{definition}

\theoremstyle{definition}

\theoremstyle{definition}

\theoremstyle{definition}

\newcommand{\R}{\ensuremath{\mathbb{R}}}
\newcommand{\N}{\ensuremath{\mathbb{N}}} 

\def\p{\partial}
\def\i{\infty}

\def\diam{\emph{diam}} 

\def\fr{\p}

\def\a{\alpha}

\def\d{\delta}
\def\D{\Delta} 
\def\e{\epsilon}

\def\H{\mathcal{H}}

\def\Id{\mathrm{Id}}

\makeatletter
\newcommand\footnoteref[1]{\protected@xdef\@thefnmark{\ref{#1}}\@footnotemark}
\makeatother

\def\XXint#1#2#3{{\setbox0=\hbox{$#1{#2#3}{\int}$}
\vcenter{\hbox{$#2#3$}}\kern-.5\wd0}}

\renewcommand{\theenumi}{(\alph{enumi})}
\makeatletter
\renewcommand{\p@enumii}{\theenumi}
\makeatother                                                

\begin{document} 
	
\author[H. Pugh]{H. Pugh\\ Mathematics Department \\ Stony Brook University} 
\title{A Localized Besicovitch-Federer Projection Theorem}
\begin{abstract}
	The classical Besicovitch-Federer projection theorem implies that the \( d \)-dimensional Hausdorff measure of a set in Euclidean space with non-negligible \( d \)-unrectifiable part will strictly decrease under orthogonal projection onto almost every \( d \)-dimensional linear subspace. In fact, there exist maps which are arbitrarily close to the identity in the \( C^0 \) topology which have the same property. A converse holds as well, yielding the following rectifiability criterion: under mild assumptions, a set is rectifiable if and only if its Hausdorff measure is lower semi-continuous under bounded Lipschitz perturbations.
\end{abstract}
\maketitle

\section{Introduction}

Rectifiable sets are the higher-dimensional analogs of rectifiable curves, and play a deep role in geometric measure theory and the calculus of variations. A set \( E\subset \R^n \) is \emph{\( m \)-rectifiable} if there exists a sequence of Lipschitz maps \( f_i: \R^m\to \R^n \) such that the \( m \)-dimensional Hausdorff measure \( \H^m \) of the set \( E\setminus \cup_i f_i(\R^m) \) is zero. An \( m \)-rectifiable set \( E\in \R^n \) has a well-defined \( m \)-dimensional tangent plane \( \H^m \) almost everywhere. With an appropriate notion of orientation and multiplicity, rectifiable sets sets can be integrated against differential forms, and these \emph{(integer) rectifiable currents} satisfy a powerful compactness theorem \cite{federerfleming}.

Rectifiable sets also arise naturally in the study of densities of measures. Let \( s \) be a non-negative real number. The \( s \)-density of a Radon measure \( \mu \) on \( \R^n \) at a point \( p\in \R^n \) is defined to be the quantity \( \Theta^s(\mu, p)=\lim_{r\to 0} \mu(B(p,r))/(2r)^s \), where \( B(p,r) \) denotes the ball of radius \( r \) about \( p \). If \( \Theta^s(\mu,p) \) exists, and is positive and finite for \( \mu \) almost all \( p\in \R^n \), then \( s \) is an integer \cite{marstrand}, and moreover \( \mu<< \H^s \), and \( \mu(\R^n\setminus E)=0 \) for some \( s \)-rectifiable Borel set \( E\subset \R^n \) \cite{preiss}. Conversely, if \( \mu \) is the measure \( B\mapsto \H^m(E\cap B) \) for some \( m \)-rectifiable Borel set \( E \) with \( \H^m(E)<\i \), then \( \Theta^m(\mu,p)=1 \) for \( \mu \) almost every \( p\in \R^n \).

Lastly, an \( \H^m \) measurable set \( E \) with \( \H^m(E)<\i \) is \( m \)-rectifiable if and only if for every \( \H^m \) measurable subset \( A \) of \( E \) with \( \H^m(A)>0 \), the set of linear \( m \)-planes \( K \) for which the orthogonal projection of \( A \) onto \( K \) is an \( \H^m \) null set, is itself null in the usual measure on the Grassmannian of linear \( m \)-planes in \( \R^n \). This is the Besicovitch-Federer projection theorem. Roughly speaking, the \( m \)-rectifiable sets are those which almost always cast an \( m \)-dimensional shadow. 

Such characterizations of rectifiability are exceedingly useful. Preiss's theorem \cite{preiss} is a central tool in proving regularity of solutions to variational problems; the rectifiability of weak limits of approximate solutions can be recast in terms of the often more tractable problem of computing densities (e.g. \cite{lipschitz}.) However, there are situations in which densities cannot be computed directly via monotonicity formulae, such as in the case of non-isotropic integrands (e.g. \cite{almgrenannals}, \cite{elliptic},) and it is more useful to have a variational approach.

The main theorem of this paper, Theorem \ref{thm:rect}, is a variatonal characterization of rectifiability: under mild assumptions, a set is rectifiable if and only if its Hausdorff measure is lower semi-continuous under bounded Lipschitz perturbations.

Theorem \ref{thm:rect} is closely related to Almgren's analytic criterion for rectifiability of quasiminimal sets \cite{almgrenannals} 3.2(c): We assume less about the set \( E \), however (namely that it is semi-regular, whereas Almgren assumes that \( E \) has certain bounds on its upper density which are related to its quasiminimality constant.) However, we have no doubt that Almgren would have been aware of this result, because it is a straightforward application of Theorem \ref{thm:main}, a version of which is stated without proof in \cite{almgrenannals} 2.9(b). An alternative proof of Theorem \ref{thm:main} can be found in \cite{feuvrier}.

We shall prove Theorem \ref{thm:rect} using a modification of the classical Federer-Fleming projection theorem in which a set \( E \) with \( \H^m(E)<\i \) can be pushed onto an arbitrarily small \( m \)-dimensional grid without increasing the Hausdorff measure of \( E \) by more than a bounded amount. We modify this construction using the Besicovitch-Federer projection theorem so that in addition, a given purely \( m \)-unrectifiable\footnote{A set is purely \( m \)-unrectifiable if its intersection with every \( m \)-rectifiable set is \( \H^m \) null.} subset of \( E \) is sent to a \( \H^m \) null set.

\section{Main Result}

Notation and terminology will follow \cite{mattila}.
\begin{thm}
	\label{thm:main}
	Suppose \( M \) is a riemannian manifold of dimension \( n \) and that \( E\subset M \) is \( \H^m \) measurable with \( \H^m(E)<\i \) and \( \H^{m+1}(\bar{E})=0 \), where \( m \) is an integer between \( 0 \) and \( n \). Suppose \( U\subset E \) is \( \H^m \) measurable and purely \( m \)-unrectifiable. For each \( \e>0 \) there exists a Lipschitz map \( \psi_\e: M\to M \) which can be uniformly approximated by diffeomorphisms of \( M \) in the isotopy class of the identity, such that
	\begin{enumerate}
		\item \( \|\psi_\e-\Id\|_0<\e \);
		\item \( \psi_\e\equiv \Id \) away from an \( \e \)-neighborhood of \( U \);
		\item \( \H^m(\psi_\e(E\setminus U))<\H^m(E\setminus U)+\e \); and
		\item \( \H^m(\psi_\e(U))<\e \).
	\end{enumerate}
\end{thm}

In particular, if \( \e>0 \) is chosen small enough so that \( \e<\H^m(U)/4 \), then \[ \H^m(\psi_\e(E))<\H^m(E) - \H^m(U)/2. \] If \( E \) is semi-regular\footnote{See \cite{davidsemmes} Definition 3.29. A closed subset \( E \) of \( \R^n \) is \emph{semi-regular of dimension \( m \)} if there exists \( C<\i \) such that for each \( x\in \R^n \) and each \( 0<r\leq R \) the set \( E\cap B(x,R) \) can be covered by \( C r^{-m}R^m \) balls of radius \( r \). Ahlfors regular sets are semi-regular.}, then \( \psi_\e \) can be chosen so that its Lipschitz constant is bounded independently of \( \e \). A partial converse to Theorem \ref{thm:main} is the following:
\begin{thm}
	\label{thm:main2}
	Suppose \( M \) is a riemannian manifold of dimension \( n \) and that \( E\subset M \) is \( \H^m \) measurable with \( \H^m(E)<\i \). Suppose \( \eta>0 \) and that there exists a sequence of Lipschitz maps \( \left(\psi_i: M\to M \right)_{i\in \N} \) whose Lipschitz constants have an upper bound, such that \( \|\psi_i-\Id\|_0 \to 0 \) and
	\[ \H^m(\psi_i(E))<\H^m(E)-\eta \] for all \( i\in \N \). Then \( E \) is not \( m \)-rectifiable.
\end{thm}

It is straightforward to check that the assumption on the Lipschitz constants of the maps \( \psi_i \) is necessary. Combining Theorems \ref{thm:main} and \ref{thm:main2},
\begin{thm}
	\label{thm:rect}
	If \( E\subset M \) is closed and semi-regular with \( \H^m(E)<\i \), then \( E \) is \( m \)-rectifiable if and only if for every sequence of Lipschitz maps \( \left(\psi_i: M\to M \right)_{i\in \N} \) such that \( \|\psi_i-\Id\|_0 \to 0 \) and whose Lipschitz constants have an upper bound, it holds that \( \liminf \H^m(\psi_i(E))\geq \H^m(E) \).
\end{thm}

Theorem \ref{thm:main} will follow from Lemma \ref{lem:FF}, which is a generalization of \cite{elliptic} Lemma 3.03, which is itself a perturbation of \cite{davidsemmes} Proposition 3.1, a version of the classical Federer-Fleming projection theorem \cite{federerfleming} 5.5. Note that a version of Lemma \ref{lem:FF} is stated without proof in \cite{almgrenannals} (2.9 (b1.)) Given a (closed) \( n \)-cube \( Q\subset \R^n \), a set \( S\subset Q \) and \( j \ge 0 \), let \( \D_j(Q,S) \) denote the collection of \( n \)-cubes in the \( j \)-th dyadic subdivision of \( Q \) which intersect \( S \) non-trivially. For \( 0 \le d \le n \) let \( \D_{j,d}(Q,S) \) denote the collection of the \( d \)-dimensional faces of the \( n \)-cubes in \( \D_j(Q,S) \) and let \( S_{j,d}(Q,S)\subset Q \) denote the union of these faces.

\begin{lem}
	\label{lem:FF}
	Suppose that \( E\subset \R^n \) satisfies \( \H^{m+1}(\bar{E}) = 0 \), that \( W\subset E \) is \( \H^m \) measurable, purely \( m \)-unrectifiable, \( \H^m(W)<\i \), and \( \bar{W}\subset \ring{S} \), and that \( R\subset E \) satisfies \( \H^m(R)<\i \). For \( 0\leq j <\i \) large enough, there exists a Lipschitz map \( \phi: \R^n \to \R^n \) with the following properties:
	\begin{enumerate}
		\item\label{lem:ff:item:0} \( \phi \) can be uniformly approximated by diffeomorphisms which fix \( \R^n\setminus S_{j,n}(Q,S) \);
		\item\label{lem:ff:item:1} \( \phi = Id \text{ on } \R^n\setminus S_{j,n}(Q,S) \);
		\item\label{lem:ff:item:2} \( \phi(E\cap S_{j,n}(Q,S)) \subset S_{j,m}(Q,S) \cup \fr S_{j,n}(Q,S) \);
		\item\label{lem:ff:item:3} \( \phi(T) \subset T \) for each \( T \in \D_j(Q,S) \);
		\item\label{lem:ff:item:4} \( \H^m(\phi(W))=0 \); and
		\item\label{lem:ff:item:5} \( \H^m(\phi(R \cap T)) \le C \H^m(R \cap T) \) for all \( T \in \D_j(Q,S) \), where \( C \) depends only on \( n \);
		\item\label{lem:ff:item:6} If \( E \) is semi-regular, then the Lipschitz constant of \( \phi \) depends only on \( n \) and the semi-regularity constant of \( E \).
	\end{enumerate}
\end{lem}

The case that \( S=Q \) and \( R=E=\bar{E} \) is Lemma \cite{elliptic} Lemma 3.03\footnote{The proof of Lemma 3.03 differs from that of \cite{davidsemmes} Proposition 3.1 in that the usual radial projections involved in the construction of the map are approximated by ambient diffeomorphisms, by which unrectifiability of \( W \) is preserved, after which nearly orthogonal projections onto \( S_{j,m}(Q,Q) \) are selected via the Besicovitch-Federer projection theorem and preformed, whereby the unrectifiable part vanishes.}. For \( S\subsetneq Q \), the map \( \phi \) is modified so that its constituant projections occur only within the \( d \)-cubes in \( \D_{j,d}(Q,Q) \) which have non-trivial intersection with the interior of \( S_{j,n}(Q,S) \), for \( m<d\leq n \). The general case that the inclusions \( R\subset E\subset \bar{E} \) are possibly strict follows from the observation that in \cite{davidsemmes} Lemma 3.22, the set \( F \) need not be closed. The projection points are chosen so that the measure of the image of \( R\cap T \), instead of \( E\cap T \), is controlled.

\begin{proof}[Proof of Theorem \ref{thm:main}]
	Using charts, the general case will follow from the case \( M=\R^n \). By \cite{mattila} Theorem 6.2, we may assume without loss of generality that \( {\Theta^*}^m(E\setminus U,x)=0 \) and \( 2^{-m}\leq {\Theta^*}^m(U,x)\leq 1 \) for all \( x\in U \). Thus for each \( \d>0 \) and \( x\in U \) there exists a sequence \( r_i\to 0 \) such that \[ \H^m((E\setminus U)\cap B(x,r_i))< \d r_i^m \] and \[ r_i^m/2\H^m(U\cap B(x,r_i))\leq 2^{m+1}r_i^m. \]

	By \cite{mattila} Theorem 2.8 there exist a finite set of points \( \{ x_1,\dots, x_N \} \subset U \) and radii \( \{ r_1,\dots,r_N \}\subset \R^+ \) such that the balls \( B(x_k,r_k) \) are pairwise disjoint, 
	\begin{equation}
		\label{eq:upper0}
		\H^m(U\setminus \cup_{k=1}^N B(x_k,r_k))<\d,
	\end{equation} 
	and the inequalities
	\begin{equation}
		\label{eq:upper1}
		\H^m((E\setminus U)\cap B(x_k,r_k))< \d r_k^m,
	\end{equation}
	and
	\begin{equation}
		\label{eq:upper2}
		r_k^m/2< \H^m(U\cap B(x_k,r_k))
	\end{equation}
	hold for each \( k\in \{1,\dots,N\} \). By \eqref{eq:upper1} and \eqref{eq:upper2},
	\begin{equation}
		\label{eq:upper2.5}
		\sum_{k=1}^N \H^m((E\setminus U) \cap B(x_k, r_k))< 2 \d \H^m(E).
	\end{equation}
	
	By outer regularity, there exists \( \eta>0 \) such that the balls \( B(x_k,r_k+\eta) \) remain pairwise disjoint, and such that
	\begin{equation}
		\label{eq:upper3}
		\sum_{k=1}^N \H^m(E\cap B(x_k,r_k+\eta)\setminus B(x_k,r_k))<\d.
	\end{equation}
	
	Now we apply Lemma \ref{lem:FF}: Let \( Q\subset \R^n \) be a cube large enough so that \( Q\supset \cup_{k=1}^N B(x_k,r_k+\eta) \) and let \( S=\cup_{k=1}^N B(x_k,r_k+\eta/2) \). Let \( W=U\cap \left(\cup_{k=1}^N B(x_k,r_k)\right) \) and let \( R=E\setminus W \). Produce from Lemma \ref{lem:FF} the map \( \phi \) for \( j \) large enough so that 
	\begin{equation}
		\label{eq:upper4}
		S_{j,n}(Q,S)\subset \cup_{k=1}^N B(x_k,r_k+\eta)
	\end{equation}
	and
	\begin{equation}
		\label{eq:upper5}
		2^{-j}\diam(Q)<\e.
	\end{equation}
	Let \( \psi_\e \) be this map \( \phi \). It follows from \eqref{eq:upper5} that \( \|\psi_\e-\Id\|_0<\e \). Moreover by \eqref{eq:upper0} and Lemma \ref{lem:FF} \ref{lem:ff:item:1}, \ref{lem:ff:item:4}, and \ref{lem:ff:item:5},
	\begin{align*}
		\H^m(\psi_\e(U))&\leq \H^m(U\setminus S_{j,n}(Q,S)) + \H^m(\psi_\e(W)) + \H^m(\psi_\e(U\cap S_{j,n}(Q,S)\setminus \cup_{k=1}^N B(x_k,r_k) ))\\
		&\leq \d + 0 + C \H^m(R\cap S_{j,n}(Q,S)).
	\end{align*}
	By \eqref{eq:upper4}, \eqref{eq:upper2.5} and \eqref{eq:upper3},
	\begin{equation}
		\label{eq:upper5.5}
		\H^m(R\cap S_{j,n}(Q,S)) \leq \sum_{k=1}^N \H^m((E\setminus U )\cap B(x_k,r_k))+ \H^m(E\cap B(x_k,r_k+\eta) \setminus B(x_k,r_k)) < 2 \d \H^m(E) + \d,
	\end{equation}
	thus
	\begin{equation*}
		\H^m(\psi_\e(U)) < \d(1+ 2C \H^m(E) + C).
	\end{equation*}
	Likewise, by \eqref{eq:upper5.5} and Lemma \ref{lem:FF} \ref{lem:ff:item:1} and \ref{lem:ff:item:5},
	\begin{align*}
		\H^m(\psi_\e(E\setminus U))&\leq \H^m(E\setminus U\setminus S_{j,n}(Q,S)) + \H^m(\psi_\e(R\cap S_{j,n}(Q,S) ))\\
		&\leq \H^m(E\setminus U) + C \H^m(R\cap S_{j,n}(Q,S))\\
		&< \H^m(E\setminus U) + \d (2C \H^m(E)+ C).
	\end{align*}
	The result follows by setting \( \d = \e / (1+ 2C \H^m(E) + C) \).
\end{proof}

\begin{proof}[proof of Theorem \ref{thm:main2}]
	Again we may assume \( M=\R^n \) by working in charts. Suppose \( E \) is \( m \)-rectifiable. By \cite{mattila} Theorem 15.21 and outer regularity, we may assume without loss of generality that \( E \) is contained in a finite family of compact embedded \( m \)-dimensional \( C^1 \) submanifolds with boundary \( M_1,\dots,M_N \) of \( M \), such that \( \H^m(\cup M_j \setminus E)<\e \), for any \( \e>0 \). 
	
	Since the Lipschitz constants of the maps \( \psi_i \) have an upper bound, we may assume further without loss of generality that in fact \( E=\cup M_j \). The result follows from lower semicontinuity of Hausdorff measure for sequences with the uniform concentration property (see \cite{dms, moso}, also \cite{davidsingular} \S 35): 
	
	Let \( 0<\e<1/2 \) and suppose \( x\in E\setminus (\cup \p M_j) \). Fix \( k \) so that \( x\in M_k \) and let \( T_x M_k\subset \R^n \) denote the affine \( m \)-plane through \( x \) tangent to \( M_k \) at \( x \). There exists \( r(x)<\textrm{dist}(x,\cup \p M_j)/2 \) small enough so that within \( B(x, 2r(x)) \), the manifold \( M_k \) is a graph of some function \( g \) over \( B(x,2r(x))\cap T_x M_k \), and such that \( M_k\cap B(x,2r(x)) \) is contained in the tangent cone\footnote{By this we mean the set of points \( \{ y\in \R^n : \mathrm{dist}(y,T_x M_k)\leq \e \|y-x\| \} \).} of \( M_k \) of aperture \( \e \) about \( x \). Let \( \pi:\R^n\to \R^n \) denote orthogonal projection onto \( T_x M_k \). If \( 0<r\leq r(x) \), then for \( i \) large enough, the set \( \pi(\psi_i(M_k)\cap B(x,r)) \) contains \( B(x,(1-\e)r)\cap T_x M_k \): if not, let \[ \rho: T_x M_k\setminus \{p\} \to \p B(x,r) \cap T_x M_k \] denote the radial projection away from a point \( p \) of \[ (B(x,(1-\e)r)\cap T_x M_k) \setminus \pi(\psi_i(M_k)\cap B(x,r)). \] For \( i \) large enough, the map \( f\equiv \rho\circ\pi\circ\psi_i \) is defined on the neighborhood \( K \) of \( x \) in \( M_k \) given by the graph of \( g \) over \( B(x,3/2 r)\cap T_x M_k \), yet restricts to a degree \( 1 \) map \( f\lfloor_{\p K}: \p K\simeq S^{m-1} \to \p B(x,r) \cap T_x M_k \simeq S^{m-1}  \), a contradiction. 
	
	Therefore, \( \H^m(\psi_i(E)\cap B(x,r))\geq \H^m(B(x,(1-\e)r)\cap T_x M_k)= (1-\e)^m r^m \a_m r^m \), where \( \a_m \) denotes the \( m \)-dimensional Hausdorff measure of the unit ball in \( \R^m \), and lower-semicontinuity holds by Theorem 10.14 \cite{moso} (or Theorem 35.4 \cite{davidsingular},) a contradiction.
\end{proof}

\addcontentsline{toc}{section}{References} 
\bibliography{bibliography.bib, mybib.bib}{}
\bibliographystyle{amsalpha}
\end{document}